\documentclass[reqno,centertags,12pt]{amsart}
\usepackage[letterpaper,margin=1.3in]{geometry}
\usepackage{amsmath,amsthm,amsfonts,amssymb,enumerate}
\usepackage[bookmarksopen=false,final]{hyperref}
\usepackage[utf8]{inputenc}
\usepackage{xcolor}

\newtheorem{theorem}{Theorem}

\theoremstyle{definition}
\newtheorem{definition}[theorem]{Definition}

\numberwithin{equation}{section}
\numberwithin{theorem}{section}

\begin{document}
	\title{Unbounded Widom factors for orthogonal and residual polynomials}
	
	\author{G\"{o}kalp Alpan}
	\address{Faculty of engineering and natural sciences,
		Sabancı University, İstanbul, Turkey}
	\email{gokalp.alpan@sabanciuniv.edu}

	
	\begin{abstract}
        We study Widom factors for (a) monic orthogonal polynomials in $L^2$ with respect to the equilibrium measure of a compact set $K\subset\mathbb{R}$ and (b) residual polynomials normalized at an exterior point. Using weakly equilibrium Cantor sets $K(\gamma)$, we prove:

(1) Given any sequence $(c_n)$ with subexponential growth, there exists a non-polar Cantor set $K(\gamma)$, depending on $(c_n)$, such that the $L^2$ Widom factors of the associated orthogonal polynomials (with respect to the equilibrium measure of $K(\gamma)$) exceed $c_n$ for every $n$.

(2) For the same $K(\gamma)$ built from $(c_n)$ and each exterior point $x_0\in\mathbb{R}\setminus K(\gamma)$, the residual Widom factors satisfy power-type lower bounds with a Harnack–distance exponent $\tau_{x_0}\in(0,1)$: they are bounded below by $c_n^{\tau_{x_0}}$ for all degrees when $x_0$ lies in an unbounded gap, and along a subsequence of degrees when $x_0$ lies in a bounded gap. Consequently, if, in addition, $(c_n)$ is monotone increasing and unbounded, then the sequence of residual Widom factors is unbounded for every $x_0\in\mathbb{R}\setminus K(\gamma)$.

The proofs combine inverse-image constructions, capacity comparisons for period-$n$ sets, harmonic-measure representations for differences of Green functions, and alternation principles on nested approximants.
	\end{abstract}
	
	\date{\today}
	\maketitle
	
	\section{Introduction}
\label{Sec:1}

Let $K\subset\mathbb C$ be a compact set that contains infinitely many points and let $n\in\mathbb N$. Let $\|\cdot\|_K$ denote the sup-norm on $K$. Then the unique monic polynomial $T_{n,K}$ which satisfies
$$\|T_{n,K}\|_K=\inf\{\|P_n\|_K: P_n \mbox{ monic of degree } n\}$$
is called the $n$-th Chebyshev polynomial on $K$; see, e.g., \cite[Theorem 7.5.6]{Davis} for uniqueness. We refer the reader to \cite{NSZ21, SodYud} for expository articles on Chebyshev polynomials. 

Many asymptotic properties and inequalities involving Chebyshev polynomials (and some other extremal polynomials that we discuss in this paper) are expressed in terms of concepts from logarithmic potential theory. Standard references for logarithmic potential theory are \cite{Ran95, ST97}.

Let $\operatorname{Cap}(K)$ denote the logarithmic capacity of $K$. Then (see, e.g., \cite[Theorem 5.5.4 (a)]{Ran95}) 
\begin{equation}\label{firsteq}
\|T_{n,K}\|_K\geq \operatorname{Cap}(K)^n.
\end{equation}
Under the additional assumption that $K\subset \mathbb{R}$, we have (see \cite{Sch08})
\begin{equation*}
\|T_{n,K}\|_K\geq 2 \operatorname{Cap}(K)^n.
\end{equation*}
It is a classical result that (see \cite[Corollary 5.5.5]{Ran95})
\begin{equation}\label{thirdeq}
\lim_{n\rightarrow\infty}\|T_{n,K}\|_K^{1/n}=\operatorname{Cap}(K).
\end{equation}
Assume that $K$ is additionally non-polar, i.e., $\operatorname{Cap}(K)>0$. We define the $n$-th Widom factor for the sup-norm on $K$ by 
\begin{equation*}
W_{\infty,n}{(K)}:= \frac{\|T_{n,K}\|_K}{\operatorname{Cap}(K)^n}. 
\end{equation*}
The term Widom factor was coined by Goncharov and Hatinoğlu in \cite{GonHat15}. Widom factors for the sup-norm (both the above version and corresponding to the weighted sup-norm problem) have been studied extensively, especially in recent years.
It follows from \eqref{firsteq} that they are bounded below by $1$. Depending on $K$, the Widom factors may converge, have finitely many limit points, or even fill a non-trivial compact interval (see \cite[Theorem 8.4]{Wid69} for examples). 

For a non-polar compact set $K\subset \mathbb C$, let $\mu_K$ denote its equilibrium measure. Let $\Omega_K$ denote the connected component of $\overline{\mathbb C}\setminus K$ that contains the point $\infty$. Let $g_K$ denote the Green function with a pole at infinity on $\Omega_K$ (see \cite[Eq. (I.4.8)]{ST97}), 
\begin{equation}\label{grnf}
		g_K(z) = -\log\operatorname{Cap}(K)+\int\log|z-\zeta|\,d\mu_K(\zeta), \quad z\in\mathbb C.
\end{equation}
If $\Omega_K$ is regular with respect to the Dirichlet problem and $\sum_j g_K(c_j)<\infty$, where the $c_j$ are the critical points of $g_K$ in $\Omega_K\setminus\{\infty\}$, then $\Omega_K$ is called a Parreau-Widom domain. If $\Omega_K$ is a Parreau-Widom domain and $K\subset\mathbb{R}$, then $K$ is called a Parreau-Widom set. These sets are important in the modern theory of Widom factors, as they form a broad class for which many results are available. If $K$ is a Parreau-Widom set then $\left(W_{\infty,n}{(K)}\right)_{n=1}^\infty$ is bounded above, see \cite[Theorem 1.4]{CSZ17}. Note that there are certain Cantor sets (by a Cantor set we mean a totally disconnected and perfect compact set)  that are Parreau-Widom, see \cite[p. 125]{PehYud03} and \cite[Sec. 7]{AlpGon17}. We should also note that a Parreau-Widom set has positive Lebesgue measure, see \cite[p. 406]{SodYud2}. This is relevant to our work because the Cantor sets that we consider in Section 2 have the same Cantor set construction used in \cite[Sec. 7]{AlpGon17} but here we choose the parameters so that the resulting Cantor set has Lebesgue measure zero. Indeed, this was also the approach used in \cite{GonHat15}.

We say that a sequence ${(c_n)}_{n=1}^\infty$ with $c_n\geq 1$ for $n\in\mathbb N$ has (at most) subexponential growth if $\frac{\log c_n}{n}\rightarrow 0$ as $n\rightarrow\infty$. The limit result \eqref{thirdeq} can be rewritten as $\lim_{n\rightarrow\infty} {W_{\infty,n}{(K)}}^{1/n}=1$, so combining this and \eqref{firsteq}, we see that the sequence of Widom factors on a non-polar compact set $K$ for the sup-norm has subexponential growth. Despite this theoretical constraint, it was shown in \cite[Theorem 4.4]{GonHat15} that for every sequence $(c_n)_{n=1}^\infty$ with subexponential growth, there is a non-polar Cantor set $K(\gamma)$ such that $W_{\infty,n}{(K(\gamma))}>c_n$ and, in particular, they gave many examples of unbounded Widom factors for the sup-norm. 

Our main objective is to extend this result to the $L^2$ case and to obtain a weaker analogue for residual polynomials; we also discuss these notions in the introduction below. We use the same Cantor set construction to prove these results, see Section 2 below.

We refer the reader to \cite{Alp22, AZ20b, AZ24, AndNaz18, CR24, CER24, CSYZ19, CSZ17, CSZ3, CSZ4, CSZ6, GonHat15, Sch08, SZ21, Tot11, Tot14, Wid69} for a more detailed account of results on Widom factors for the sup-norm.

For a finite positive Borel measure $\mu$ with compact non-polar support $\operatorname{supp}(\mu)$ in $\mathbb C$, we can define monic orthogonal polynomials as follows. Let $n\in\mathbb N$. Then the unique monic polynomial $Q_n(\cdot;\mu)$ satisfying
\begin{equation}\label{min ort}
\|Q_n(\cdot;\mu)\|_{L^2(\mu)}=\inf\{\|P_n\|_{L^2(\mu)}: P_n \mbox{ monic of degree } n\}
\end{equation} 
is called the $n$-th monic orthogonal polynomial for $\mu$. 
For such a measure $\mu$, we define the $n$-th Widom factor for $\mu$ as 

\begin{equation*}
W_{2,n}{(\mu)}:= \frac{\|Q_n(\cdot;\mu)\|_{L^2(\mu)}}{\operatorname{Cap}(\operatorname{supp}\mu)^n}. 
\end{equation*}

For all $n$, we have (see \cite[Corollary 1.3]{Alp19})

\begin{equation}\label{lower equi}
W_{2,n}{(\mu_K)}\geq 1.
\end{equation}

Let $L=\operatorname{supp}(\mu_K)$. By the definition of the equilibrium measure, $L\subset K$, and we also have (see \cite[Lemma 1.2.7]{ST92})
\begin{equation}\label{eq cap}
\operatorname{Cap}(L)= \operatorname{Cap}(K).
\end{equation}

 Since $\mu_K$ is a unit measure and the norm of the $n$-th Chebyshev polynomial increases on larger sets, comparing the $L^2$ and $L^\infty$ norms of $T_{n,L}$ and using minimality of orthogonal polynomials \eqref{min ort}, we obtain

\begin{equation}\label{long ineq}
\|T_{n,K}\|_K \geq \|T_{n,L}\|_L \geq \|T_{n,L}\|_{L^2(\mu_K)} \geq \|Q_n(\cdot;\mu_K)\|_{L^2(\mu_K)}.
\end{equation}
In view of \eqref{eq cap}, the inequality \eqref{long ineq} implies that
\begin{equation}\label{ineq wid l2}
W_{\infty,n}(K)\geq W_{2,n}(\mu_K).
\end{equation}
Combining \eqref{ineq wid l2}, \eqref{lower equi} and \eqref{thirdeq}, we see that $(W_{2,n}(\mu_K))^{1/n}\rightarrow 1$ as $n\rightarrow\infty$ (originally this was proved in \cite{Wid67}), and therefore $(W_{2,n}(\mu_K))_{n=1}^\infty$ has subexponential growth. 

The first main result of the paper is Theorem \ref{t1}, where we show that for any sequence $(c_n)_{n=1}^\infty$ with subexponential growth there is a non-polar Cantor set $K(\gamma)$ such that $W_{2,n}(\mu_{K(\gamma)})\geq c_n$. Note that this result does not follow from \cite{GonHat15} because of \eqref{ineq wid l2}.

We refer the reader to \cite{Alp22, AZ20a, AZ20b, AZ24, Chr12, CSZ11} for more on Widom factors for orthogonal polynomials.

Let $K$ be a non-polar compact subset of $\mathbb R$. Then $\mathbb{R}\setminus K$ can be written as a disjoint union of its open connected components $(\alpha_j, \beta_j)$, called the gaps of $K$.  Let $x_{-}:=\inf K$ and $x_+:= \sup K$. Then $\mathbb R\setminus [x_{-}, x_+]$ contains two unbounded gaps $(-\infty,x_{-})$ and $(x_+,\infty)$. We call the remaining gaps bounded gaps of $K$. Let $x_0\in \mathbb R\setminus K$ and $n\in \mathbb N$. Then the $n$-th residual polynomial, denoted by $R_{n,K}^{(x_0)}$, is the unique polynomial (see \cite[Theorem 2.1]{CSZ5} for uniqueness) of degree \emph{at most} $n$ that minimizes $\|P\|_K$ among all polynomials of degree at most $n$ satisfying $P(x_0)=1$. Under these conditions, $\deg R_{n,K}^{(x_0)}$ is either (see \cite[Theorem 3.7(b)]{CSZ5}) $n$ or $n-1$. 

We have (see, e.g., \cite[p.~14]{Sch17} or \cite{dris, Kuij, Sch11, CSZ5}, as these results appear many times in the literature)

\begin{equation}\label{res eq1}
\lim_{n\rightarrow\infty} \|R_{n,K}^{(x_0)} \|_K^{1/n}= e^{-g_K(x_0)},
\end{equation}
and
\begin{equation}\label{res eq2}
\|R_{n,K}^{(x_0)} \|_K\geq e^{-ng_K(x_0)}.
\end{equation}
Following \cite{CSZ6}, we define the $n$-th Widom factor for the residual polynomials at $x_0$ on $K$ by
\begin{equation*}
W_{\infty,n}^{(x_0)}(K)=\frac{\|R_{n,K}^{(x_0)} \|_K} {e^{-ng_K(x_0)}}.
\end{equation*}
Thus, in view of \eqref{res eq1} and \eqref{res eq2}, we have $W_{\infty,n}^{(x_0)}(K)\geq 1$ and $(W_{\infty,n}^{(x_0)}(K))_{n=1}^\infty$ has subexponential growth. In addition, $W_{\infty,n}^{(x_0)}(K)\rightarrow W_{\infty,n}(K)$ as $x_0\rightarrow\infty$ (see \cite[Eq. (1.16)]{CSZ6}). Hence, it is natural to view Widom factors for residual polynomials and the sup-norm in a unified way as in \cite{buch, CSZ6}. For more on Widom factors of the residual polynomials, we refer the reader to \cite{buch, CSZ5, CSZ6}. We refer the reader to \cite{char, dris, ech, Kuij} for other applications of residual polynomials.

Our second main result is Theorem \ref{t2}. Given a monotone increasing sequence $(c_n)_{n=1}^\infty$ with subexponential growth, there exists a non-polar Cantor set $K(\gamma)$ such that for each exterior point $x_0\in\mathbb{R}\setminus K(\gamma)$ the residual Widom factors admit power-type lower bounds with an exponent $\tau_{x_0}\in(0,1)$ (specified later): if $x_0$ lies in an unbounded gap of $K(\gamma)$, then $W_{\infty,n}^{(x_0)}(K(\gamma))$ is bounded below by $c_n^{\tau_{x_0}}$ for all $n$; if $x_0$ lies in a bounded gap, then there exists a subsequence $n_k\to\infty$ such that $W_{\infty,n_k}^{(x_0)}(K(\gamma))$ is bounded below by $c_{n_k}^{\tau_{x_0}}$. This complements the $L^2$ result above and quantifies the growth via a geometric exponent.

In Section~2 we present the construction of the Cantor sets $K(\gamma)$, prove Theorem~\ref{t1}, and then prove Theorem~\ref{t2} on residual Widom factors, covering all degrees when $x_0$ lies in an unbounded gap, and degrees $n=2^s$ (for large $s$) when $x_0\in(0,1)\setminus K(\gamma)$.
\section{Unbounded Widom factors on $K(\gamma)$}
Throughout we assume that $K\subset \mathbb R$ is a compact non-polar set. Define
\begin{equation}\label{kn def}
K_n=\{z\in \mathbb C: T_{n,K}(z)\in [-\|T_{n,K}\|_K,\|T_{n,K}\|_K ]  \}
  \end{equation}
Then (see \cite[Section 2]{CSZ17}) $K_n$ is a subset of the real line, $K\subset K_n$, and $\|T_{n,K}\|_K= 2\operatorname{Cap}(K_n)^n \geq 2\operatorname{Cap}(K)^n.$ An immediate consequence of this is the expression 
\begin{equation}\label{somein}
W_{\infty,n}(K)=\frac{2\operatorname{Cap}(K_n)^n}{\operatorname{Cap}(K)^n}
\end{equation}

A compact set $S\subset \mathbb R$ is called a finite gap set if $S=\cup_{k=1}^j [a_k, b_k]$, where $j\geq 1$ and the sets $[a_k, b_k]$ are non-trivial compact disjoint intervals. A finite gap set $S$ is called a period-$n$ set if each connected component of $S$ has equilibrium measure $k/n$ for some $k\in\{1,\ldots, n\}$. There are many characterizations of these sets, see \cite[Propositions 1.1 and 1.3]{Peh1} and  \cite[Theorem 5.5.25]{Sim11}. Note that $K_n$ in \eqref{kn def} is always a period-$n$ set, see \cite[Theorem 2.5]{CSZ17}.


Let $x_0\in \mathbb R\setminus K$. Then a polynomial $P$ of degree at most $n$ has an $x_0$-alternating set on $K$ if there exist $n+1$ points $x_1,\ldots,x_{n+1}$ in $K$ such that
\[
x_1<x_2<\cdots<x_k<x_0<x_{k+1}<\cdots<x_{n+1}
\]
for some $k\in\{1,\ldots,n+1\}$ and
\begin{equation}
P(x_j)=(-1)^{k+1-j}\,\operatorname{sign}(x_j-x_0)\,\|P\|_K,\quad j=1,\ldots,n+1.
\end{equation}
By \cite[Theorem~3.5]{CSZ5}, $R_{n,K}^{(x_0)}$ has an $x_0$-alternating set on $K$, and if $P$ is a polynomial of degree at most $n$ that has an $x_0$-alternating set on $K$, with $P(x_0)=1$, then $P=R_{n,K}^{(x_0)}$.

Let $|\cdot|$ denote the Lebesgue measure of a set on $\mathbb{R}$. 
We recall the construction of $K(\gamma)$ from \cite{Gonc14}. Let $\gamma:=(\gamma_s)_{s=1}^\infty$ be a sequence such that $0<\gamma_s<1/4$ for all $s\in \mathbb N$. Let $r_0:=1$ and define
\begin{equation}\label{rs def}
r_s:= \gamma_s r_{s-1}^2= \gamma_s \gamma_{s-1}^2\gamma_{s-2}^4\cdots \gamma_1^{2^{s-1}}
\end{equation}
recursively. Let $P_2(x)= x(x-1)$ and $P_{2^{s+1}}(x):=P_{2^s}(x)(P_{2^s}(x)+r_s)$ be defined recursively as well. Let $E_0=[0,1]$.
For each $s\in \mathbb N$, define
\begin{equation}\label{inv eq}
E_s:= \left(\frac{2}{r_s}P_{2^s}+1\right)^{-1}{([-1,1])}= \cup_{j=1}^{2^s} I_{j,s}
\end{equation}
where $I_{j,s}$ are disjoint closed non-trivial intervals in $[0,1]$ ordered from left to right. Here, $I_{2j-1,s+1}$ and $I_{2j,s+1}$ are proper subsets of $I_{j,s}$, and by \cite[Lemma 2]{Gonc14},\\ $\max_{1\leq j\leq 2^s} |I_{j,s}|\rightarrow 0$ as $s\rightarrow\infty$. In addition, $E_{s}\subset E_{s-1}$ for each $s\in \mathbb N$ and therefore $K(\gamma):=\bigcap_{s=0}^\infty E_s$ is a Cantor set and $\{0,1\}\subset K(\gamma)$ regardless of the choice of $\gamma$. 

Note that $T_{2^s,K(\gamma)}=P_{2^s}+r_s/2$, see \cite[Proposition 1]{Gonc14}. Here $\|T_{2^s,K(\gamma)}\|_{K(\gamma)}=r_s/2$, and 
\begin{equation}\label{es def2}
E_s= \{z\in \mathbb C: T_{2^s,K(\gamma)}(z)\in [-\|T_{2^s,K(\gamma)}\|_{K(\gamma)},\|T_{2^s,K(\gamma)}\|_{K(\gamma)} ]  \}. 
\end{equation}
So $E_s$ is a period-$2^s$ set containing $K(\gamma)$. If $x_0\in  (0,1)\setminus K(\gamma)$ then there exists $s_0\in \mathbb N$ such that $x_0\in E_{s_0-1}$ but $x_0\not\in E_{s}$ for all $s\geq s_0$. Then, in view of \cite[Example 4.6]{CSZ5} and \eqref{inv eq}, for each $s\geq s_0$, we have 
\begin{equation}\label{some new}
R_{2^s, E_s}^{(x_0)}= \left(\frac{2}{r_s}P_{2^s}(x_0)+1\right)^{-1}\cdot\left(\frac{2}{r_s}P_{2^s}+1\right)=\frac{T_{2^s,K(\gamma)}}{T_{2^s,K(\gamma)}(x_0)}.
\end{equation}
Here, at each point $x_k$ ($k=1,\ldots, 2^s+1$) in the $x_0$-alternating set on $E_s$, we have $\left(\frac{2}{r_s}P_{2^s}(x_k)+1\right)=\pm 1$. The solutions of the last equation are just the endpoints of $I_{j,s}$ for $j=1,\ldots, 2^s$. These endpoints are also elements of $K(\gamma)$ by the Cantor set construction. Hence, the $x_0$-alternating set $\{x_k\}_{k=1}^{2^s+1}$ for $E_s$ is also an $x_0$-alternating set for $K(\gamma)$. Therefore $R_{2^s, E_s}^{(x_0)}=R_{2^s, K(\gamma)}^{(x_0)}$ holds, and by \eqref{some new}, 
\begin{equation}\label{another some new}
R_{2^s, K(\gamma)}^{(x_0)}=\frac{T_{2^s,K(\gamma)}}{T_{2^s,K(\gamma)}(x_0)}.
\end{equation}

Since $\operatorname{Cap}([-1,1])=1/2$, in view of \cite[Theorem 5.2.5]{Ran95}, \eqref{inv eq} and \eqref{rs def}, we get
\begin{equation}\label{cap es}
\operatorname{Cap}(E_s)=\left(\frac{r_s}{4}\right)^{1/2^s}= {\left(\frac{1}{4}\right)}^{1/2^s}\gamma_1^{1/2}\gamma_2^{1/4}\cdots \gamma_s^{1/2^s}.
\end{equation}
Since $E_s\downarrow K(\gamma)$, it follows that $\lim_{s\rightarrow\infty} \operatorname{Cap}(E_s)= \operatorname{Cap}(K(\gamma))$ in view of \cite[Theorem 5.1.3 (a)]{Ran95}. Using \eqref{cap es}, we see that
\begin{equation*}
\operatorname{Cap}(K(\gamma)) =\exp\left[\sum_{n=1}^\infty 2^{-n}\log{{\gamma_n}}\right].
\end{equation*}
Therefore $K(\gamma)$ is non-polar if and only if 
\begin{equation*}
\exp\left[\sum_{n=1}^\infty 2^{-n}\log{{\gamma_n}}\right]>0.
\end{equation*}
In addition to non-polarity, it was assumed in  \cite[Section 5]{AlpGon16} that $\gamma_s\leq 1/6$ for all $s$, so that $K(\gamma)$ and Widom factors for orthogonal polynomials associated with $\mu_{K(\gamma)}$ satisfy additional properties. Under these assumptions, (see \cite[Remark 4.8]{AlpGon16}) $|K(\gamma)|=0$, and for any $s\in \mathbb N\cup \{0\}$ (see \cite[Eq. (5.3)]{AlpGon16})
\begin{equation}\label{wid weak 1}
 W_{2,2^s}(\mu_{K(\gamma)})=\frac{\sqrt{1-2\gamma_{s+1}}}{2 \exp[\sum_{k=s+1}^\infty 2^{s-k}\log{\gamma_k}]}.
\end{equation}

For any $n\in \mathbb N$, let $s$ be the integer in $\mathbb N\cup \{0\}$  satisfying $2^s\leq n<2^{s+1}.$ Then (see the proof of \cite[Theorem 5.1 (a)]{AlpGon16})
\begin{equation}\label{wid weak 2}
\min_{2^s\leq n<2^{s+1}} W_{2,n}(\mu_{K(\gamma)})= W_{2,2^s}(\mu_{K(\gamma)}).
\end{equation}
Indeed, these assumptions lead to unbounded Widom factors, see \cite[Theorem 5.1]{AlpGon16} but our aim is much more than that as we want to control how quickly they get larger. 

Now, we verify some facts involving sequences with subexponential growth. Let $(c_n)_{n=1}^\infty$ be a sequence with subexponential growth. Adding $e-1$ to each term yields another sequence with subexponential growth. Thus, we may assume $c_n\ge e$. 

Assume $c_n\ge e$ and $\frac{\log c_n}{n}\to0$. Then we can construct ${(s_n)}_{n=1}^\infty$ with:
(i) $s_n\ge c_n$, (ii) $(s_n)_{n=1}^\infty$ is monotone increasing, (iii) $\frac{\log s_n}{n}\downarrow 0$. This is already in the proof of \cite[Lemma 4.3]{GonHat15}, but we expand the argument for the reader’s convenience.
Let $u_n:=\sup_{1\le k\le n} c_k$. Then $u_n\ge c_n$ and $(u_n)_{n=1}^\infty$ is monotone increasing. For any $\varepsilon>0$ choose $K$ with $c_k\le e^{\varepsilon k}$ for $k\ge K$, and set $M:=\max_{1\le k<K} c_k$. For $n\ge K$,
\[
u_n=\sup_{k\le n}c_k\le \max\{M,e^{\varepsilon n}\}\ \Rightarrow\ \frac{\log u_n}{n}\le \max\!\Big\{\frac{\log M}{n},\varepsilon\Big\}.
\]
Since $\varepsilon$ is arbitrary, $\frac{\log u_n}{n}\rightarrow 0$.
Write $u_n=e^{n\alpha_n}$ for $n\geq 1$ with $\alpha_n:=\frac{1}{n}\log u_n\to0$.
Let $\alpha_n^*:=\sup_{k\ge n}\alpha_k$ so $\alpha_n^*\downarrow0$ and define
\[
s_n:=\sup_{1\le k\le n} e^{k\alpha_k^*}.
\]

Then
(i) $s_n\ge e^{n\alpha_n^*}\ge e^{n\alpha_n}=u_n\ge c_n$. \quad
(ii) Monotonicity follows because enlarging the domain of the supremum never lowers its value.
 \quad
(iii) With $t_n:=\frac{1}{n}\log s_n=\frac{1}{n}\max_{k\le n}k\alpha_k^*$, we have $t_{n+1}\le\max\{\tfrac{n}{n+1}t_n,\alpha_{n+1}^*\}\le t_n$, so $(t_n)_{n=1}^\infty$ is monotone decreasing. Given $\varepsilon>0$, pick $N$ with $\alpha_N^*<\varepsilon$ and set $C:=\max_{k\le N}k\alpha_k^*$. For $n\ge N$,
$t_n\le \max\{C/n,\varepsilon\}$. Thus $\frac{1}{n}\log s_n=t_n\downarrow 0$.

In view of this discussion, we give a name to sequences such as $(s_n)_{n=1}^\infty$.
\begin{definition}
We say that a sequence $(c_n)_{n=1}^\infty$ is a regular sequence with subexponential growth if
$(i)$ $c_n\ge e$, $(ii)$ $(c_n)_{n=1}^\infty$ is monotone increasing, $(iii)$ $\frac{\log c_n}{n}\downarrow 0$.
\end{definition}
Now, we are ready to prove our first result.
\begin{theorem}\label{t1}
Let $(c_n)_{n=1}^\infty$, with $c_n\geq 1$ for each $n$,  be a sequence with subexponential growth. Then there exists $K(\gamma)$ that is non-polar and 
\begin{equation}\label{lower sub 1}
W_{2,n}(\mu_{K(\gamma)})\geq c_n\quad\text{for all } n\in\mathbb{N}.
\end{equation}

\end{theorem}
\begin{proof}
Without loss of generality, we can assume that ${(c_n)}_{n=1}^\infty$ is a regular sequence with subexponential growth.

For each $n\in \mathbb N$, let 
\begin{equation}\label{gamma def}
\gamma_n=\frac{c_{2^{n+1}}}{6 c_{2^n}^2}
\end{equation} 
By our regularity assumption, $\frac{\log{c_{2^{n+1}}}}{2^{n+1}}\leq \frac{\log{c_{2^n}}}{2^n}$. Hence $c_{2^{n+1}}^{1/2^{n+1}}\leq c_{2^{n}}^{1/2^{n}}$ and $c_{2^{n+1}}\leq c_{2^n}^2$. Substituting this into \eqref{gamma def}, we see that 
\begin{equation*}
\gamma_n\le \frac{1}{6}\qquad\text{for all }n.
\end{equation*}
For any $s\in \mathbb{N}$, substituting \eqref{gamma def} into \eqref{cap es}, we get
\begin{equation}\label{long form}
\displaystyle\operatorname{Cap}(E_s)=\left(\frac{1}{4}\right)^{1/2^s}\left(\frac{1}{6}\right)^{\sum_{k=1}^s (1/2)^k}\frac{c_{2^{s+1}}^{1/2^s}}{c_2}.
\end{equation}
Since $(c_n)_{n=1}^\infty$ has subexponential growth, $c_{2^{s+1}}^{1/2^s}\rightarrow 1$. Hence,  using \eqref{long form}, we deduce that
\begin{equation*}
\operatorname{Cap}(K(\gamma))=\lim_{s\rightarrow\infty} \operatorname{Cap}(E_s)= \frac{1}{6c_2}>0.
\end{equation*}
Hence $K(\gamma)$ is non-polar and $\gamma_n\leq 1/6$ for all $n$. So, we can use both \eqref{wid weak 1} and \eqref{wid weak 2}.

Now, it remains to prove \eqref{lower sub 1}.
Let $s\in\mathbb{N}\cup\{0\}.$ Since $\gamma_{s+1}\leq 1/6$, we have $(1-{2\gamma_{s+1}})/4\geq 1/6.$ Hence, using this in \eqref{wid weak 1} and squaring both sides, we see that
\begin{equation}\label{mid wid ineq}
\left[W_{2,2^s}(\mu_{K(\gamma)})\right]^2\geq \frac{1}{6}\cdot\frac{1}{\exp{\left[\sum_{k=s}^\infty 2^{s-k}\log{\gamma_{k+1}}\right]}}.
\end{equation}
For $m>s$, in view of \eqref{gamma def}, we have
\begin{align}
\exp{\left[\sum_{k=s}^m 2^{s-k}\log{\gamma_{k+1}}\right]}&=\gamma_{s+1}\cdot\gamma_{s+2}^{1/2}\cdots \gamma_{m+1}^{1/2^{m-s}}\\
&=\displaystyle\left( \frac{1}{6}\right)^{\sum_{k=0}^{m-s} {2^{-k}}}\frac{\left(c_{2^{m+2}} \right)^{1/2^{m-s}}}{c_{2^{s+1}}^2}.\label{ryr}
\end{align}
Fixing $s$ and letting $m\rightarrow\infty$, we see that $\left(c_{2^{m+2}} \right)^{1/2^{m-s}}\rightarrow 1$ since $(c_n)_{n=1}^\infty$ has subexponential growth. Now, using this in \eqref{ryr}, while fixing $s$ and letting $m\rightarrow \infty$, we obtain
\begin{equation*}
\exp{\left[\sum_{k=s}^\infty 2^{s-k}\log{\gamma_{k+1}}\right]}=\frac{1}{36\cdot c_{2^{s+1}}^2}.
\end{equation*}
Substituting this into \eqref{mid wid ineq}, we deduce that
\begin{equation*}
\left[W_{2,2^s}(\mu_{K(\gamma)})\right]^2\geq 6c_{2^{s+1}}^2,
\end{equation*}
and thus
\begin{equation}\label{finalish wid}
W_{2,2^s}(\mu_{K(\gamma)})\geq \sqrt{6}c_{2^{s+1}}.
\end{equation}
Fix $n\in\mathbb N$ and let $s$ be the integer satisfying $2^s\leq n<2^{s+1}$.
Combining \eqref{wid weak 2}, \eqref{finalish wid} and the assumption that $(c_n)_{n=1}^\infty$ is monotone increasing, we see that
\begin{equation}\label{rrrr1}
W_{2,n}(\mu_{K(\gamma)})\geq W_{2,2^s}(\mu_{K(\gamma)})\geq \sqrt{6} c_{2^{s+1}}\geq \sqrt{6} c_n,
\end{equation}
which implies \eqref{lower sub 1} and completes the proof.
\end{proof}

For any domain $\Omega\subset \overline{\mathbb C}$ and $z,w\in \Omega$, the Harnack distance between $z$ and $w$ is the smallest number $\tau_\Omega(z,w)$ such that for every (strictly) positive harmonic function $h:\Omega\rightarrow \mathbb R$,
\begin{equation*}
(\tau_\Omega(z,w))^{-1} h(w)\leq h(z)\leq \tau_\Omega(z,w) h(w)
\end{equation*}
holds. It exists, satisfies $\tau_\Omega(z,w)= \tau_\Omega(w,z)$, and $\tau_\Omega(z,w)\geq 1$ since positive constants are harmonic, see \cite[Section 1.3]{Ran95}. In addition, $\tau_\Omega(z,w)\rightarrow 1$ as $z\rightarrow w$, see the proof of \cite[Theorem 1.3.8]{Ran95}.

For any domain $\Omega\subset \overline{\mathbb C}$ containing $\infty$ with non-polar compact boundary $\partial \Omega$ in $\mathbb R$, we denote the harmonic measure defined at $z\in \Omega$ by $\omega_\Omega (\cdot; z)$. In this case, the equilibrium measure of the boundary is given by $\mu_{\partial \Omega}(\cdot)= \omega_\Omega (\cdot; \infty)$, see \cite[Theorem 4.3.14]{Ran95}. Moreover, $\mu_{\partial \Omega}(\cdot)$ and $\omega_\Omega (\cdot; z)$ are mutually absolutely continuous, see \cite[Corollary 4.3.5]{Ran95}.
Now, we are ready to prove our second main result.

\begin{theorem}\label{t2}
Let $(c_n)_{n=1}^\infty$ be a regular sequence with subexponential growth. Then there is a non-polar Cantor set $K(\gamma)$ such that
\begin{enumerate}[$(i)$]
\item if $x_0$ lies in an unbounded gap of $K(\gamma)$, then for all $n\in\mathbb{N}$,
\begin{equation}\label{lower sub 11}
W_{\infty,n}^{(x_0)}(K(\gamma))\geq c_n^ {{\tau_{x_0}}} 
\end{equation}
holds,
\item if $x_0\in(0,1)\setminus K(\gamma)$, then there exists an integer $s_0$ such that for all $s\ge s_0$,
\begin{equation}\label{lower sub 111}
W_{\infty,2^s}^{(x_0)}(K(\gamma))\geq c_{2^s}^ {{\tau_{x_0}}} 
\end{equation}
holds,
\end{enumerate}
where we define $\tau_{x_0}$ as follows: Let $L_{x_0}$ be the connected component of $\mathbb R\setminus K(\gamma)$ containing $x_0$. Let $\Omega_{x_0}:= (\overline{\mathbb C}\setminus [0,1]) \cup L_{x_0}$. Define $\tau_{x_0}:= \frac{1}{\tau_{\Omega_{x_0}}(x_0,\infty)}.$
In particular, if $(c_n)_{n=1}^\infty$ is unbounded then there exists a non-polar Cantor set $K(\gamma)$ such that $\sup_n W_{\infty,n}^{(x_0)}(K(\gamma))=\infty$ holds for all $x_0\in\mathbb R\setminus K(\gamma)$.
\end{theorem}
\begin{proof}
As in the proof of Theorem \ref{t1}, for each $n\in \mathbb N$, let 
\begin{equation}\label{gamma def2}
\gamma_n=\frac{c_{2^{n+1}}}{6 c_{2^n}^2}.
\end{equation} 
So $K(\gamma)$ is non-polar and $|K(\gamma)|=0.$

Let $x_0\in\mathbb R\setminus K(\gamma)$ and $n\geq 2$.
Let 
\begin{equation}\label{fdef}
F_n:= \{ z\in \mathbb C: R_{n,K(\gamma)}^{(x_0)}(z)\in [-\|R_{n,K(\gamma)}^{(x_0)}\|_{K(\gamma)},\|R_{n,K(\gamma)}^{(x_0)}\|_{K(\gamma)}] \}
\end{equation}
Then (see \cite[Proposition 3.10]{CSZ5}), $F_n$ is a period-$d_n$ set where $d_n:= \deg{R_{n,K(\gamma)}^{(x_0)}}$, the gap $L_{x_0}$ containing $x_0$ (and also $x_0$ itself) does not intersect with $F_n$ and $K(\gamma)\subset F_n$. For $n=1$, if $x_0$ is in a bounded gap of $K(\gamma)$ then $R_{1,K(\gamma)}^{(x_0)}\equiv 1$, see, e.g., \cite[Example 4.3]{CSZ5}. So, we do not define $F_1$ for such $x_0$. If $x_0$ is in an unbounded gap then $\deg R_{1,K(\gamma)}^{(x_0)}=1$ in view of \cite[Theorem 3.6]{CSZ5} so $F_1$ is defined via \eqref{fdef} by substituting $n=1$ there. Since $F_n$ is a finite gap set, $\operatorname{supp}(\mu_{F_n})= F_n$ (see, e.g., \cite[Eq. (2.16)]{CSZ17} for the exact formulation of the measure). Since $|K(\gamma)|=0$, in view of \cite[Proposition 2.1]{CSZ3}, we have $\operatorname{Cap}(K(\gamma)) < \operatorname{Cap}(F_n)$. 

For $n\geq 1$, define
\begin{equation*}
K_n=\{z\in \mathbb C: T_{n,K(\gamma)}(z)\in [-\|T_{n,K(\gamma)}\|_{K(\gamma)},\|T_{n,K(\gamma)}\|_{K(\gamma)} ]  \}.
\end{equation*}

Arguing as in the proof of \cite[Proposition 4.1]{CSZ3}, set
\[
H_n(z):=g_{K(\gamma)}(z)-g_{F_n}(z).
\]
Then $H_n$ is harmonic on $\mathbb{C}\setminus F_n$. Since $g_{K(\gamma)}$ and $g_{F_n}$ have the same logarithmic behavior at infinity, $H_n$ is bounded near $\infty$ and in fact admits the finite limit
\[
H_n(\infty)=\log\!\Big(\tfrac{\operatorname{Cap}(F_n)}{\operatorname{Cap}(K(\gamma))}\Big).
\]
Thus $H_n$ extends to a bounded harmonic function on $\overline{\mathbb{C}}\setminus F_n$. Moreover, by \eqref{grnf} (with $K=F_n$ and $K=K(\gamma)$), both $g_{F_n}$ and $g_{K(\gamma)}$ are continuous nearly everywhere (n.e.) in $\mathbb{C}$ (i.e., outside a Borel polar set); see \cite[Theorem I.4.4]{ST97} and \cite[Theorems 4.2.5 and 4.4.9]{Ran95}. Therefore, the n.e. boundary values of $H_n$ on $F_n$ equal $g_{K(\gamma)}-g_{F_n}$; since $g_{F_n}=0$ n.e. on $F_n$ and finite unions of polar sets are polar \cite[Corollary 3.2.5]{Ran95}, $H_n$ has n.e. boundary values equal to $g_{K(\gamma)}$ on $F_n$. Hence, by \cite[Corollary 4.2.6 and Theorem 4.3.3]{Ran95}, for any $z\in\overline{\mathbb{C}}\setminus F_n$,
\begin{equation}\label{long eq}
g_{K(\gamma)}(z)-g_{F_n}(z)=H_n(z)=\int g_{K(\gamma)}(t)\, d\omega_{\Omega_{F_n}}(t;z).
\end{equation}

Substituting $z=\infty$ in \eqref{long eq}, we exactly get (see \cite[Proposition 4.1]{CSZ3}),
\begin{equation}\label{log exp}
H_n(\infty)=\log{\left(\frac{\operatorname{Cap}(F_n)}{\operatorname{Cap}(K(\gamma))} \right)}=\int g_{K(\gamma)}(t)\,d\mu_{F_n}(t)
\end{equation}

Substituting $z=x_0$ in \eqref{long eq}, we get
\begin{equation}\label{yet another}
H_n(x_0)=g_{K(\gamma)}(x_0)-g_{F_n}(x_0).
\end{equation}

Since the integrand in \eqref{long eq} is positive on a set of $\mu_{F_n}$-measure 1, it is also positive on a set of harmonic measure 1 for any $z$, by mutual absolute continuity. Hence $H_n$ is a positive harmonic function on $\overline{\mathbb C}\setminus F_n$.

We first deal with the unbounded gap case. Assume that $x_0\in \mathbb R\setminus [0,1]$. Then by \cite[Theorem 3.6]{CSZ5}, $R_{n,K(\gamma)}^{(x_0)}=T_{n,K(\gamma)}/{T_{n,K(\gamma)}(x_0)}$ is satisfied. Therefore for all $n\geq 1$, we have
\begin{equation}\label{inter iden}
K_n=F_n\qquad\text{if }x_0\in\mathbb{R}\setminus[0,1].
\end{equation}

and so in view of \cite[Eq. (1.3)]{CSZ17}, this implies that
\begin{equation*}
F_n\subset[0,1]\qquad\text{if }x_0\in\mathbb{R}\setminus[0,1].
\end{equation*}

Since $F_n\cap L_{x_0}=\emptyset$, this implies that for all $n\in\mathbb N$, we have
\begin{equation*}
\Omega_{x_0}=(\overline{\mathbb C}\setminus[0,1])\cup L_{x_0}=(\overline{\mathbb C}\setminus[0,1]) \subset \overline{\mathbb C}\setminus F_n.
\end{equation*}
By the monotonicity property of Harnack distance (see \cite[Corollary 1.3.7]{Ran95}), this implies
\begin{equation}\label{harineq}
\tau_{\overline{\mathbb C}\setminus F_n}(x_0,\infty)\leq \tau_{\Omega_{x_0}}(x_0,\infty).
\end{equation}

Note that since $R_{n,K(\gamma)}^{(x_0)}=T_{n,K(\gamma)}/{T_{n,K(\gamma)}(x_0)}$ holds for all $n$, we have $d_n=n$ for this case. Using Harnack’s inequality in \eqref{longeq1}, the inequality \eqref{harineq} in \eqref{longeq2}, the identity \eqref{log exp} in \eqref{longeq3}, and \eqref{inter iden} in \eqref{longeq4}, and then rewriting the logarithm via \eqref{somein} for $K=K(\gamma)$ in \eqref{somelong} and applying the definition of $\tau_{x_0}$ in \eqref{longeq5} we obtain
\begin{align}
\exp[{n H_n(x_0)}]&\geq  \exp{\left[\tau_{\overline{\mathbb C}\setminus F_n}^{-1}(x_0,\infty)n H_n(\infty)\right]}\label{longeq1}\\
&\geq \exp{\left[\tau_{\Omega_{x_0}}^{-1}(x_0,\infty)n H_n(\infty)\right]}\label{longeq2}\\
&=\exp{\left[\tau_{\Omega_{x_0}}^{-1}(x_0,\infty)n \log{\left(\frac{\operatorname{Cap}(F_n)}{\operatorname{Cap}(K(\gamma))}\right)}\right]}\label{longeq3}\\
&=\exp{\left[\tau_{\Omega_{x_0}}^{-1}(x_0,\infty)n \log{\left(\frac{\operatorname{Cap}(K_{n})}{\operatorname{Cap}(K(\gamma))}\right)}\right]}\label{longeq4}\\
&=\exp{\left[\tau_{\Omega_{x_0}}^{-1}(x_0,\infty) \log{\left(\frac{W_{\infty,n}(K(\gamma))}{2}\right)}\right]}\label{somelong}\\
&=\left(\frac{1}{2}\right)^{\tau_{x_0}}\left[W_{\infty,n}(K(\gamma))\right]^{\tau_{x_0}}\label{longeq5}
\end{align}

On the other hand, using the definition of Widom factors in \eqref{longeq11}, rewriting $\|R_{n,K(\gamma)}^{(x_0)}\|_{K(\gamma)}$ in terms of Green function (see \cite[Eq. (3.10)]{CSZ5} where we use the fact $d_n=n$ here.) in \eqref{longeq21}, using positivity of $g_{F_n}(x_0)$ in \eqref{longeq31} and \eqref{yet another} in \eqref{longeq41}, we obtain

\begin{align}
W_{\infty,n}^{(x_0)}(K(\gamma))&=\exp[ng_{K(\gamma)}(x_0)] \|R_{n,K(\gamma)}^{(x_0)}\|_{K(\gamma)} \label{longeq11}\\
&= \frac{2\exp[ng_{K(\gamma)}(x_0)]}{\exp[n g_{F_n}(x_0)]+ \exp[-n g_{F_n}(x_0)]} \label{longeq21}\\
&\geq \exp[n (g_{K(\gamma)}(x_0)-g_{F_n}(x_0))]\label{longeq31}\\
&= \exp[{n H_n(x_0)}].\label{longeq41}
\end{align}
Combining \eqref{longeq1}, \eqref{longeq5}, \eqref{longeq11} and \eqref{longeq41}, we obtain
\begin{equation}\label{ttt}
W_{\infty,n}^{(x_0)}(K(\gamma))\geq \left(\frac{1}{2}\right)^{\tau_{x_0}} \left[ W_{\infty,n}(K(\gamma)) \right]^{\tau_{x_0}},
\end{equation}

Since we use the same $\gamma$ here \eqref{gamma def2} and in the proof of Theorem \ref{t1}, it follows from \eqref{ineq wid l2} and \eqref{rrrr1} that
\begin{equation}\label{q1 e1}
W_{\infty, n}(K(\gamma))\geq W_{2,n}(\mu_{K(\gamma)})\geq  \sqrt{6} c_n.
\end{equation} 
Combining \eqref{ttt} and \eqref{q1 e1}, we obtain
\begin{equation}\label{finalcountdown}
W_{\infty,n}^{(x_0)}(K(\gamma))\geq \left(\frac{\sqrt{6} c_n}{2}\right)^ {{\tau_{x_0}}}.
\end{equation}
Since $0<\tau_{x_0}\leq 1$, the inequality \eqref{lower sub 11} follows from \eqref{finalcountdown}. 

Now, assume that $x_0\in (0,1)\setminus K(\gamma)$. Then there exists $s_0\in \mathbb N$ such that $x_0\in E_{s_0-1}$ but $x_0\not\in E_{s}$ for all $s\geq s_0$ where the sets $E_s$ are defined as in \eqref{inv eq}. As discussed above \eqref{another some new}, $R_{2^s, K(\gamma)}^{(x_0)}=\frac{T_{2^s,K(\gamma)}}{T_{2^s,K(\gamma)}(x_0)}$ holds in this case so for $s\geq s_0$,
\begin{equation*}
K_{2^s}=F_{2^s}\qquad\text{if }x_0\in(0,1)\setminus K(\gamma).
\end{equation*}
So
\begin{equation*}
F_{2^s}\subset[0,1]\qquad\text{if }x_0\in(0,1)\setminus K(\gamma).
\end{equation*}

and since $F_n\cap L_{x_0}=\emptyset$, this implies that for all $s\geq s_0$, we have
\begin{equation*}
\Omega_{x_0}=((\overline{\mathbb C}\setminus [0,1]) \cup L_{x_0} ) \subset \overline{\mathbb C}\setminus F_{2^s}.
\end{equation*}
By the monotonicity property of the Harnack distance, this implies
\begin{equation*}
\tau_{\overline{\mathbb C}\setminus F_{2^s}}(x_0,\infty)\leq \tau_{\Omega_{x_0}}(x_0,\infty).
\end{equation*}
So all the inequalities above hold for this case when we substitute $n=2^s$. That is, using \eqref{longeq1}, \eqref{longeq5}, \eqref{longeq11}, \eqref{longeq41} for $n=2^s$, we get 
\begin{equation}\label{o3}
W_{\infty,2^s}^{(x_0)}(K(\gamma))\geq \left(\frac{1}{2}\right)^{\tau_{x_0}} \left[ W_{\infty,2^s}(K(\gamma)) \right]^{\tau_{x_0}},
\end{equation}
Substituting $n=2^s$ in \eqref{q1 e1}, we obtain
\begin{equation}\label{o4}
W_{\infty, 2^s}(K(\gamma))\geq W_{2,2^s}(\mu_{K(\gamma)})\geq  \sqrt{6} c_{2^s}.
\end{equation}
Combining \eqref{o3} and \eqref{o4}, for $s\geq s_0$, we get
\begin{equation}\label{finalcountdown2}
W_{\infty,2^s}^{(x_0)}(K(\gamma))\geq \left(\frac{\sqrt{6} c_{2^s}}{2}\right)^ {{\tau_{x_0}}}.
\end{equation}
Since $0<\tau_{x_0}\leq 1$, the inequality \eqref{lower sub 111} follows from \eqref{finalcountdown2}. 

Since we assume $c_n$ is monotone increasing, if $(c_n)_{n=1}^\infty$ is unbounded then both \eqref{lower sub 11} and \eqref{lower sub 111} imply that $\sup_{n\in \mathbb N} W_{\infty,n}^{(x_0)}(K(\gamma))=\infty$ holds and this completes the proof.

\end{proof}

\section*{Acknowledgment}
The author thanks the American Institute of Mathematics for the hospitality and the stimulating
environment of the SQuaRE program.

	
\end{document}